\documentclass[12pt]{article}

\usepackage{amssymb}
\usepackage{amsmath,color}

\usepackage[T1]{fontenc} 
\usepackage{lmodern}\usepackage{graphicx}
\usepackage{microtype} 

\usepackage[hmarginratio=1:1,top=32mm,columnsep=20pt]{geometry} 

\usepackage{amsthm}
\usepackage{url}

\newtheorem{thm}{Theorem}[section]
\newtheorem{cor}[thm]{Corollary}

\newtheorem{prop}[thm]{Proposition}
\newtheorem{defn}[thm]{Definition}

\title{\vspace{-15mm}\fontsize{24pt}{10pt}\selectfont\textbf{New equivalences to
axioms weaker than \textbf{AC} in topology}} 

\author{
\large
\textsc{Daniel de la Concepci\'on}\thanks{FPU12 Grant from the Ministerio de Educaci\'on, Cultura y Deporte in Spain}\\[2mm] 
\normalsize Universidad de la Rioja \\
Departamento de Matem\'aticas y Computaci\'on\\ 
\vspace{-5mm}
}
\date{}

\begin{document}

\maketitle 


\begin{abstract}
In this work, new equivalences of topological statements and weaker axioms
than \textbf{AC} are proven. This equivalences include the use of anti-properties.
All this equivalences have been checked with a computer using the theorem proving system
Isabelle/Isar and are available at the isarmathlib repository.
\end{abstract}

\section{Introduction}

The close relation between set theory and general topology makes it easy for different equivalences between several
axioms independent of $\textbf{ZF}$ and several theorems in general topology to appear.

An example of this relation lies, for example, in the fact that several topological properties are defined in terms of cardinals such as
the Lindelöf property or the first or second countable properties which are related to the first uncountable cardinal, and the relation between them
depend on the relations between the cardinals which may vary drastically choosing a different set
of axioms.

In this paper, we prove the equivalence between axioms which are weaker than ${\textbf{AC}}$ and statements that consider only topological properties. This has been done already in several papers, like \cite{Ho96,Ho02,Ho97} and the references therein. The difference with the present paper is that we include the concept of anti-property which has been studied only considering $\bf{ZFC}=\bf{ZF}+\bf{AC}$ since this concept is defined over the assumption that every set is bijective with some cardinal. Readers can check the expansion of this theories in~\cite{Br98His}, an historical survey from 1998, and the references therein.

\subsection{Definitions and notation}

In \textbf{ZF} the classes of sets are defined as predicates in one variable. For convenience, we will write a class as the collection of sets that make this predicate hold. In consequence, we will write $A\in P$ when $P(A)$; and $P\subseteq Q$ when $\forall A,\ P(A)\Rightarrow Q(A)$.

Also, we consider the cardinality of the sets that are bijective with some ordinal. We define this cardinal as the least ordinal, in the order of ordinals, bijective to that set. The rest of the sets do not have a cardinality. When we write $|A|$, it implies that $A$ has a cadinality.

Under this construction, the class of cardinals is totally ordered. Hence, we will denote by $\kappa^+$ as the succesor cardinal of the cardinal $\kappa$.
We consider also $\aleph_0$ the first infinite cardinal and $\aleph_{\alpha+1}$ as $\aleph_\alpha^+$.

\begin{defn}
A topological space is said to be $\kappa$-Lindelöf, for a cardinal $\kappa$ iff every open cover
$\mathcal{U}$ has a subcover $\mathcal{V}$ with $|\mathcal{V}|<\kappa$.
\end{defn}

A few classical examples of the previous properties are that $\aleph_0$-Lindelöf $=$ Compact and $\aleph_1$-Lindelöf $=$ Lindelöf.

\begin{defn}
A topological space $(X,\tau)$ is said to have weight smaller than $\kappa$, for a cardinal $\kappa$ iff there is a base $B$ with $|B|<\kappa$.
\end{defn}

\begin{defn}\cite{Ban79}
A topological property is a property of topological spaces that is preserved by homeomorphisms.

Given a topological property $P$, we defined its spectrum $Spec(P)$ as the class of sets such that
any topology defined over them has the property $P$.

Then anti-$P$ is defined as: A topological space is anti-$P$ iff the only subspaces that have the property $P$ are those in its spectrum.
\end{defn}

A classical example of this concept is: $A\in Spec(Compact)$ iff $A$ is finite; and hence the anti-compact spaces are those that appear in the literature as pseudo-finite spaces. Other examples are anti-perfect $=$ scattered or anti-connected $=$ totally disconnected.

Let's compute what anti-hyperconnected means to give an idea of the methods involved in this theory:

\begin{defn}
A topological space is hyperconnected iff every pair of non-empty open sets intersect.

A topological space is sober iff every non-empty hyperconnected subspace is the closure of a point
and that point is unique.
\end{defn}

\begin{thm}
A topological space is anti-hyperconnected iff it is $T_1$ and sober.
\end{thm}

\begin{proof}
The first step in our proof is to compute the spectrum of hyperconnection.

Assume that $A$ has more than $1$ point, then the discrete topology on $A$
gives a non hyperconnected space. Hence $A$ is not in the spectrum.

If $A$ has no more than $1$ point, the only topology that can be defined in it
is the discrete one. This topology has no pair of non-empty open sets, and
then is trivially hyperconnected.

The conclusion is: $A\in Spec(Hyperconnection)$ iff $A$ has no more than $1$ point.

Let's proof next an auxiliary result: The closure of a point is always hyperconnected.

Consider $(X,\tau)$ a topological space and $x\in X$. Consider then, the subspace given by the subset $cl(\{x\})$.
Take any $U,V\in \tau$ such that $U_s=U\cap cl(\{x\})\ne\emptyset$ and $V_s=V\cap cl(\{x\})\ne\emptyset$. By definition of closure, $y\in cl(\{x\})$ iff every open set that contains $y$ also contains $x$. In conclusion, since $U_s\ne\emptyset$ then $x\in U$. The same happens for $V$. Then $x\in U_s\cap V_s$ and it follows that $U_s\cap V_s\ne\emptyset$. Hence the space $cl(\{x\})$ is hyperconnected.

At last, let's proof the theorem:

Assume that $(X,\tau)$ is anti-hyperconnected.

Consider $x\in X$. As the previous lemma
states, $cl(\{x\})$ is hyperconnected as a subspace of $(X,\tau)$. From the definition of $Spec(Hyperconnection)$, follows that $cl(\{x\})$ has no more than one point; and so, $cl(\{x\})=\{x\}$.
In conclusion, $\{x\}$ is a closed set and hence $(X,\tau)$ is $T_1$.

Consider $Y$ a non-empty hyperconnected subspace of $(X,\tau)$. By the definition of anti-property, it follows that $Y$ has one point, $Y=\{y\}$. Now $Y=cl(\{y\})$, since the space is $T_1$ and if $cl(\{x\})=Y$ it follows that $x\in Y$ and hence $x=y$. This is the definition of a sober space.

Assume now that $(X,\tau)$ is $T_1$ and sober.

Consider $Y$ an hyperconnected subspace. Since the space is sober, $Y=cl(\{y\})$ for some $y\in Y$ or $Y=\emptyset$. Since the space is $T_1$, $Y=\{y\}$ or $Y=\emptyset$. It follows then that $Y\in Spec(Hyperconnection)$. In conclusion $(X,\tau)$ is anti-hyperconnected.
\end{proof}

Let's define some weaker forms of \textbf{AC} that will appear from now on:

\begin{defn}
${\bf C}(\kappa,S)$ for a cardinal $\kappa$ and a set $S$ means: 

For every family of sets $\{\mathcal{N}_t\}_{t\in A}\subseteq Pow(S)$ and $|A|\leq\kappa$, exists $f:A\to S$ such that $\forall t\in A, f(t)\in \mathcal{N}_t$.

Also, if $\kappa=\aleph_0$ we write ${\bf CC}(S)$. If ${\bf C}(\kappa,S)$ holds for every set $S$, then we write ${\bf C}(\kappa)$; and if ${\bf C}(\kappa,S)$ holds for every cardinal $\kappa$, then we write ${\bf AC}(S)$.
\end{defn}

It is easy to prove that if $S$ and $T$ are bijective sets, then ${\bf C}(\kappa,S)\Leftrightarrow{\bf C}(\kappa,T)$.

\subsection{Previous results}

There has been several articles on the subject of relations between weaker forms of choice and topological statements. Here there is a list of known propositions of which we'll make use and that we will try to generalize.

\begin{thm}\cite{Ho97}
Equivalent are:
\begin{enumerate}
\item ${\bf CC}(Pow(\aleph_0))$
\item $\mathbb{N}$ with the discrete topology is Lindelöf.
\item Every topological space with weight smaller than $\aleph_1$ is Lindelöf.
\item $\mathbb{R}$ is Lindelöf.
\item $\mathbb{Q}$ is Lindelöf.
\item Every unbounded subset of $\mathbb{R}$ contains un unbounded sequence.
\end{enumerate}
\end{thm}

On the anti-properties there have been also some results that are useful in our work. They study the implications of anti-properties when the original properties are related. The following are the first simple results by Bankston in his original paper.

\begin{prop}\cite{Ban79}\label{ban}
Consider $P$ and $Q$ topological properties.

\begin{enumerate}
\item If $P\Rightarrow Q$, then $Spec(P)\subseteq Spec(Q)$. 
\item If $P\Rightarrow Q$ and $Spec(P)= Spec(Q)$, then anti-$Q\Rightarrow$ anti-$P$.
\item anti-$P$ and anti-$Q$ need not to be related in general when $P\Rightarrow Q$.
\item If a topological space $(X,\tau)$ is $P$ and anti-$P$, then $X$ is in $Spec(P)$.
\end{enumerate}
\end{prop}

\section{Computations of spectra without choice}

To work with the anti-properties, it is necessary to compute the spectrum of the original property. This isn't always easy as we shall show:

\begin{thm}
If a topological property $P$ is preserved by coarser topologies, then $A$ is in $Spec(P)$ iff $(A,Pow(A))$ has the property $P$.
\end{thm}

\begin{proof}
Assume that $A$ is $Spec(P)$, then $(A,Pow(A))$ has the property $P$ by the definition of spectrum.

Assume that $(A,Pow(A))$ has the property $P$, and consider $(A,\tau)$ a topological space over $A$.
By assumption, since $\tau\subseteq Pow(A)$, it follows that $(A,\tau)$ has the property $P$. In conclusion,
$A$ is in $Spec(P)$ by definition.
\end{proof}

Applying this result to the Lindelöf family of properties:

\begin{cor}\label{eqSpecLin}
The set $A$ is in $Spec(\kappa$-Lindelöf$)$ iff $(A,Pow(A))$ is $\kappa$-Lindelöf.
\end{cor}

\begin{cor}\label{cor1}
The spectra of Lindelöf properties have the following properties:
\begin{itemize}
\item For cardinals $\kappa$ and $\mu$; if $\kappa\leq \mu$, then
 $Spec(\kappa$-Lindelöf$)\subseteq Spec(\mu$-Lindelöf$)$. 
\item Also, if $\exists f:B\to A$ injective and $A$ in $Spec(\kappa$-Lindelöf$)$, $B$ is also in $Spec(\kappa$-Lindelöf$)$.
\item And $\kappa$ is not in $Spec(\kappa$-Lindelöf$)$
\end{itemize}
\end{cor}

\begin{proof}
Since $\kappa$-Lindelöf $\Rightarrow \mu$-Lindelöf, from \ref{ban} the first statement follows.

Assume that $f:B\to A$ is injective. Then $f:(B,Pow(B))\to (A,Pow(A))$ is continuous and injective.
It can be consider then that $(B,Pow(B))$ is a closed subspace of $(A,Pow(A))$. It follows from \ref{eqSpecLin} that $(A,Pow(A))$ is $\kappa$-Lindelöf. Then $(B,Pow(B))$ is also $\kappa$-Lindelöf since $B$ is closed, and hence $B\in Spec(\kappa$-Lindelöf$)$.

Consider $\mathcal{U}=\{\{x\}|\ x\in \kappa\}$. $\mathcal{U}$ is an open cover for $(\kappa,Pow(\kappa))$ which has no subcover. Since $|\mathcal{U}|=\kappa$, it follows that $(\kappa,Pow(\kappa))$ is not $\kappa$-Lindelöf. In conclusion, $\kappa\notin Spec(\kappa$-Lindelöf$)$.
\end{proof}

To see the difficulty on computing spectra in \textbf {ZF}, let's compute $Spec($Lindelöf$)$:

This following theorem generalizes some points that appeared in~\cite{Ho97}.

\begin{thm}\label{eq1}
The following are equivalent for any infinite cardinal $\kappa$:
\begin{enumerate}
\item $\kappa$ with its discrete topology is $\kappa^+$-Lindelöf.
\item Every topological space with weight smaller than $\kappa^+$ is $\kappa^+$-Lindelöf.
\item $\textbf{C}(\kappa,Pow(\kappa))$
\end{enumerate}
\end{thm}

\begin{proof}
This proof follows the ideas of Horst Herrlich in~\cite{Ho97} for $\kappa=\aleph_0$.

Let's assume $1$ to be true and prove $2$.

Consider a topological space $(X,\tau)$ with weight smaller than $\kappa^+$. We need to prove that $(X,\tau)$ is $\kappa^+$-Lindelöf.

Consider an open cover: $\mathcal{M}\subseteq\tau$ and $\bigcup_{M\in\mathcal{M}}M=X$.
Consider also $\{B_t\}_{t\in\kappa}$ a base for $(X,\tau)$.

The application $$\begin{matrix}S&:&\mathcal{M}&\to& Pow(\kappa)\\ &&U&\mapsto&\{i\in\kappa|\ B_i\subseteq U\}\end{matrix}$$ is then injective since $U=\bigcup_{B_i\subseteq U}B_i$ by the definition of base.

$S(\mathcal{M})\subseteq Pow(\kappa)$, and hence $Y=\bigcup_{U\in S(\mathcal{M})}U$ is a closed subset of $\kappa$ with its discrete topology. By assumption, it is then $\kappa^+$-Lindelöf. Since $S(\mathcal{M})$ is an open cover of $Y$, there exists $\mathcal{V}\subseteq S(\mathcal{M})$ such that $\bigcup_{U\in \mathcal{V}}U=Y$ and $|\mathcal{V}|<\kappa^+$.

Fix $x\in X$, then $\exists R\in \mathcal{M}\subseteq\tau$ such that $x\in R$; since $\mathcal{M}$ is an open cover. There exists also $t\in\kappa$ such that $x\in B_t$ and $B_t\subseteq R$ by the definition of base. By the definition of $S$, it follows that $t\in S(R)$.
In particular $t\in \bigcup_{U\in S(\mathcal{M})}U=\bigcup_{U\in \mathcal{V}}U$. There exists then $V\in \mathcal{V}$ such that $t\in V$.
Also then $V\in S(\mathcal{M})$ and exists $T\in\mathcal{M}$ with $V=S(T)$. Then $B_t\subseteq T$ and $T\in S^{-1}(\mathcal{V})$.
In conclusion $x\in T$ and $T\in S^{-1}(\mathcal{V})$; so $X=\bigcup_{U\in S^{-1}(\mathcal{V})}U$.

By definition, $ S^{-1}(\mathcal{V})$ is a subcover of $\mathcal{M}$; but we have to prove that $| S^{-1}(\mathcal{V})|<\kappa^+$.
Since $S$ is injective, it is bijective on its image. In conclusion $|S^{-1}(\mathcal{V})|=|\mathcal{V}|<\kappa^+$.

Let's assume $2$ to be true and prove $3$.

Consider $A$ such that $|A|\leq\kappa$ and $\{N_t\}_{t\in A}\subseteq Pow(\kappa)$ such that $\forall t\in A,\ N_t\ne\emptyset$. We need to find $f:A\to Pow(\kappa)$ such that $\forall t\in A,\ f(t)\in N_t$.

Consider $\kappa\times A$ with its discrete topology. Then $B=\{\{x\}|\ x\in \kappa\times A\}$ is a base for this topology and $|B|=|\kappa\times A|$. Since $|A|\leq \kappa$, it follows that $|\kappa\times A|=|A| |\kappa|\leq \kappa^2=\kappa<\kappa^+$.

In conclusion, $\kappa\times A$ with its discrete topology has weight smaller than $\kappa^+$. Then $\kappa\times A$ with its discrete topology is $\kappa^+$-Lindelöf. Since every subspace of $\kappa\times A$ is closed, every subspace of $\kappa\times A$ is also $\kappa^+$-Lindelöf. 

Define $\mathcal{U}=\{U\times \{t\}.\ U\in N_t,\ t\in A\}$. Since $\displaystyle\bigcup_{V\in\mathcal{U}}V\subseteq \kappa\times A$, there exists $\mathcal{V}\subseteq\mathcal{U}$ such that $\displaystyle\bigcup_{V\in\mathcal{U}}V =\displaystyle\bigcup_{V\in\mathcal{V}}V$ and $|\mathcal{V}|<\kappa^+$.

Consider then an injective function $r:\mathcal{V}\to\kappa$, that exists by the last paragraph.
It follows easily that if $N_t\ne\{\emptyset\}$, there exists $k\in\kappa$ such that $(k,t)\in \displaystyle\bigcup_{V\in\mathcal{U}}V$. In particular, $(k,t)\in V$ for some $V\in\mathcal{V}$. Also, $\forall V\in\mathcal{V},\ \exists t\in A, \exists U\in N_t$ such that $V=U\times\{t\}$. In conclusion, if $N_t\ne\{\emptyset\}$, then $\{U\times\{t\}|\ U\in N_t\}\cap\mathcal{V}\ne\emptyset$.

Since $\kappa$ is a cardinal, it is well-ordered an hence $r(\{U\times\{t\}|\ U\in N_t\}\cap\mathcal{V})$ has a least element when it's not empty, let's call it $Least(t)$. Also, since $r$ is injective $r(U\times\{t\})=r(V\times\{t\})$ implies that $U=V$. In conclusion there is a unique $U\in N_t$ such that $r(U\times\{t\})=Least(t)$, let's call it $U_{least}(t)$.

Define now the following function: $$\begin{matrix}f&:&A&\to& Pow(\kappa)\\ &&t&\mapsto&\left\{\begin{matrix}\emptyset&\text{if }N_t=\{\emptyset\}\\ U_{least}(t)&\text{otherwise}\end{matrix}\right.\end{matrix}$$

The function $f:A\to Pow(\kappa)$ is the choice function since $U_{least}(t)\in N_t$ by definition and if $N_t=\{\emptyset\}$, obviously $f(t)=\emptyset\in N_t$.

Let's assume $3$ to be true and prove $1$.

Consider $\mathcal{M}\subseteq Pow(\kappa)$ such that $\displaystyle\bigcup_{M\in\mathcal{M}}M=\kappa$. We have to find $\mathcal{N}\subseteq\mathcal{M}$ so that $|\mathcal{N}|<\kappa^+$ and $\displaystyle\bigcup_{N\in\mathcal{N}}N=\kappa$.

Consider the application $$\begin{matrix}X&:&\kappa&\to& Pow(Pow(\kappa))\\ &&i&\mapsto&\{M\in\mathcal{M}|\ i\in M\}\end{matrix}$$

It is obvious that $X(i)\ne\emptyset$ for any $i\in\kappa$ since $\mathcal{M}$ covers $\kappa$.

Since $\kappa\leq\kappa$ and $\emptyset\notin X(\kappa)\subseteq Pow(Pow(\kappa))$, by assumption exists $f:\kappa\to Pow(Pow(\kappa))$ such that $f(i)\in X(i)$.

Now $f(\kappa)\subseteq \mathcal{M}$ by definition. Evenmore, $|f(\kappa)|\leq \kappa<\kappa^+$.

Consider now $i\in\kappa$, then $i\in f(i)$ by definition of $X(i)$, then $\kappa\subseteq \bigcup_{i\in\kappa} f(i)$.

In conclusion, $f(\kappa)$ is the subcover of $\mathcal{M}$ we needed; and $\kappa$ with its discrete topology is $\kappa^+$-Lindelöf.

\end{proof}

\begin{cor}\label{corpos}
Assume ${\bf C}(\kappa,Pow(\kappa))$.

Then $A$ is in $Spec(\kappa^+$-Lindelöf$)$ iff $|A|<\kappa^+$.
\end{cor}

\begin{cor}\label{corneg}
Assume $\neg {\bf C}(\kappa,Pow(\kappa))$.

Then $Spec(\kappa^+$-Lindelöf$)\subseteq \{A\in\textbf{Set}|\ |A|<\kappa\}$.
\end{cor}

\begin{cor}\label{cornegN}
If $\neg {\bf CC}(Pow(\aleph_0))$, then $Spec($Lindelöf$)=Spec($Compact$)$ is formed by the finite sets.

If ${\bf CC}(Pow(\aleph_0))$, then $Spec($Lindelöf$)$ is formed by the countable sets.
\end{cor}

The conclusion is that, the spectra vary and hence, it should follow that the anti-properties mean different things depending on different axioms we could add to $\bf{ZF}$.

\begin{thm}
$(\mathbb{N},Pow(\mathbb{N}))$ is anti-Lindelöf in $\bf{ZF}$.
\end{thm}

\begin{proof}

Assume $Y$ is a lindelöf subspace of $(\mathbb{N},Pow(\mathbb{N}))$. Let's proof that $Y$ is in $Spec($Lindelöf$)$. By the previous result \ref{eqSpecLin}
it is equivalent to prove that $(Y,Pow(Y))$ is Lindelöf. By definition, since $Y$ is a subspace of a discrete space, $Y=(Y,Pow(Y))$ and by assumption $Y$ is Lindelöf. Exactly what needed to prove.
\end{proof}

This previous statement is true independent of ${\bf AC}$ or any of its weaker axioms, but this truth means different things depending on what axioms are added to \textbf{ZF}: Every subspace of $\mathbb{N}$ is lindelöf, or only the finite subspaces of $\mathbb{N}$ are lindelöf.

\section{New equivalences with anti-properties}

In this last section, we will prove some equivalences of topological statements and axioms weaker than \textbf{AC}.

\begin{thm}\label{eq3}
The following are equivalent for an infinite cardinal $\kappa$:
\begin{enumerate}
\item ${\bf C}(\kappa,Pow(\kappa))$
\item The one point compactification of $\kappa$ with its discrete topology is anti-$\kappa^+$-Lindelöf
\item $\kappa\in Spec(\kappa^+$-Lindelöf$)$
\end{enumerate}
\end{thm}

\begin{proof}
Consider $(X,\tau)$ the one point
compactifiction of that $(\kappa,Pow(\kappa))$. 

First, notice that~\ref{eqSpecLin} makes \ref{eq1}$(1)$ and $3$ equivalent. So, $1$ and $3$ are equivalent.

Let's assume $1$ and prove $2$.

Since $|X|=|\kappa\cup\{\kappa\}|=\kappa<\kappa^+$, from~\ref{corpos}
it follows that $X\in Spec(\kappa^+$-Lindelöf$)$; and hence it is anti-$\kappa^+$-Lindelöf.

Let's assume $2$ and prove $3$.

Since $(X,\tau)$ is compact, then it is $\kappa^+$-Lindelöf. By assumption, it follows then that $X\in Spec(\kappa^+$-Lindelöf$)$.
Since $|X|=\kappa$, there is an injective function $f:\kappa\to X$; so $\kappa\in Spec(\kappa^+$-Lindelöf$)$ considering~\ref{cor1}.
\end{proof}

\begin{cor}
In case $\kappa=\aleph_0$, we can add the following to the previous equivalences:

 There is a topological space $(X,\tau)$ which is anti-Lindelöf, but not anti-compact.
\end{cor}

\begin{proof}
Assume $\neg{\bf C}(\kappa,Pow(\kappa))$ and let's show the negation of the new statement.

From~\ref{cornegN}, we get that $Spec($Lindelöf$)=Spec($Compact$)$.

Since compact $\Rightarrow$ Lindelöf; one of the results of~\ref{ban} states that
under this conditions anti-Lindelöf implies anti-compact.

Assume now the negation of new statement, and let's prove that the one point compactification of $\aleph_0$ with its discrete topology is not anti-Lindelöf.

The one point compactification is compact, then it is also Lindelöf. Assume that this space is anti-Lindelöf. By assumption, it is then anti-compact also. 

By~\ref{ban}, $Y=\aleph_0\cup\{\aleph_0\}\in Spec($Compact$)$. This is absurd, since $Y$ is not finite.
\end{proof}

This last result cannot be done for higher cardinals as easily. The reason is that $\textbf{CH}$; which we understand as $|Pow(\aleph_0)|=\aleph_1$, is independent of ${\bf ZF}+{\bf CC}$,
but the following theorem holds in $\textbf{ZF}$.

\begin{thm}
${\bf CH} \Rightarrow \aleph_0\in Spec(\aleph_2$-Lindelöf$)$.

In other words, $Spec(\aleph_2$-Lindelöf$)=Spec($Compact$)\ \Rightarrow\ \neg{\bf CH} \wedge\ \neg {\bf CC}(Pow(\aleph_0))$.
\end{thm}

\begin{proof}
Consider $\aleph_0$ with its discrete topology $Pow(\aleph_0)$. By ${\bf CH}$,  $|Pow(\aleph_0)|=\aleph_1$
and then every open cover $\mathcal{U}$ fulfills $|\mathcal{U}|\leq\aleph_1<\aleph_2$.
In conclusion, $(\aleph_0,Pow(\aleph_0))$ is trivially $\aleph_2$-Lindelöf; and hence $\aleph_0\in Spec(\aleph_2$-Lindelöf$)$.
\end{proof}

 This theorem states that if every countable set has a topological space such that it is not $\aleph_2$-Lindelöf, $\textbf{CH}$ doesn't hold. In conclusion, there are more axioms than those weaker than $\bf{AC}$ which vary the spectrum of the family of Lindelöf properties. In conclusion, every combination is possible:
 
 \begin{itemize}
 \item $Spec(\aleph_2$-Lindelöf$)=Spec(\aleph_1$-Lindelöf$)=Spec($Compact$)$
 \item $Spec(\aleph_2$-Lindelöf$)\supset Spec(\aleph_1$-Lindelöf$)=Spec($Compact$)$
 \item $Spec(\aleph_2$-Lindelöf$)=Spec(\aleph_1$-Lindelöf$)\supset Spec($Compact$)$
 \item $Spec(\aleph_2$-Lindelöf$)\supset Spec(\aleph_1$-Lindelöf$)\supset Spec($Compact$)$
 \end{itemize}

This results have deep consequences in \textbf{ZF}, as the following example shows:

\begin{thm}
In \textbf{ZF}, $(X,\tau)$ is not anti-compact iff, in \textbf{ZF}, `$(X,\tau)$ is not anti-Lindelöf' is true or undecidable.

Or in other words; $(X,\tau)$ is anti-Lindelöf in \textbf{ZF} iff `$(X,\tau)$ is anti-compact' is true or undecidable in \textbf{ZF}.
\end{thm}

This means that if there is a proof that a space is anti-Lindelöf in \textbf{ZF}, then there can be no proof within \textbf{ZF} that the space is not anti-compact. In other words, it can be proven that the space is anti-compact or more axioms are needed to prove one or the other.

\section{Isabelle/Isar}

In the Isabelle webpage~\cite{IsabIsar}, we can read the following paragraphs to know what Isabelle/Isar is:

`The Intelligible semi-automated reasoning (Isar) approach to readable formal proof documents sets out to bridge the semantic gap between internal notions of proof given by state-of-the-art interactive theorem proving systems and an appropriate level of abstraction for user-level work.'

`Isabelle is a generic proof assistant. It allows mathematical formulas to be expressed in a formal language and provides tools for proving those formulas in a logical calculus. The Isabelle system offers Isar as an alternative proof language interface layer, beyond traditional tactic scripts.'

Most of the results in this paper are available in isarmathlib~\cite{IMLdown}; a repository of
mathematical results checked in Isabelle/Isar using its \textbf{ZF} axiomatization. Some of those files are presented in a nicer looking webpage~\cite{IMLpres}.

\bibliographystyle{plainurl}
\bibliography{HorstHerrlich,antiProp,Isabelle}

\end{document}